\documentclass[a4paper,12pt]{article}
\usepackage[left=2cm,right=2cm, top=2cm,bottom=3cm,bindingoffset=0cm]{geometry}
\usepackage{cmap}
\usepackage[T2A]{fontenc}
\usepackage[cp1251]{inputenc}
\usepackage{verbatim}
\usepackage{amsmath}
\usepackage{amsthm}
\usepackage{amssymb}
\usepackage{delarray}
\usepackage{cite}
\usepackage{hyperref}
\usepackage{mathrsfs}
\usepackage{tikz}
\usepackage{dsfont}
\usetikzlibrary{patterns}
\usepackage{caption}
\newcommand{\la}{\lambda}

\DeclareCaptionLabelSeparator{dot}{. }
\theoremstyle{plain}

\numberwithin{equation}{section}

\newtheorem{thm}{Theorem}[section]
\newtheorem{lem}[thm]{Lemma}
\newtheorem{prop}[thm]{Proposition}

\theoremstyle{definition}

\newtheorem{ip}[thm]{Inverse Problem}

\theoremstyle{remark}

 \allowdisplaybreaks

\begin{document}
\begin{center}
{\large\bf Uniform stability of the inverse Sturm-Liouville problem\\[0.1cm] on a star-shaped graph}
\\[0.3cm]
{\bf Chitorkin E.E., Bondarenko N.P.} \\[0.2cm]
\end{center}

\vspace{0.5cm}

{\bf Abstract.}   In this paper, we study the inverse spectral problem for the Sturm-Liouville operators on a star-shaped graph, which consists in the recovery of the potentials from specral data or several spectra. The uniform stability of these inverse problems on the whole graph is proved.

\medskip


\medskip


\vspace{1cm}

\section{Introduction} \label{sec:intro}

Consider the star-shaped graph $G$, which consists of the vertices $\{v_j\}_{j=0}^{m}$ ($m \ge 1$) and the edges $\{ e_j \}_{j=1}^m$ of length $\pi$. Each edge $e_j$ joins the $\textit{boundary}$ vertex $v_j$ with the $\textit{internal}$ vertex $v_0$. We consider each edge $e_j$ as a segment parameterized by the corresponding variable $x_j \in [0, \pi]$. The value $x_j = 0$ corresponds to the vertex $v_j$ and $x_j = \pi$ to $v_0$.

In this paper, we consider the following Sturm--Liouville problem $L = L(\textbf{q})$ on the graph $G$:
\begin{gather} \label{eqv1}
-y_j''(x_j) + q_j(x_j) y_j(x_j) = \lambda y_j(x_j), \quad x_j \in (0, \pi), \quad j = \overline{1, m} \\ \label{bc1}
y_j(0) = 0, \quad j = \overline{1, m}, \\ \label{bc2}
y_j(\pi) = y_1(\pi), \quad j=\overline{2, m}, \\ \label{kc}
\sum\limits_{j=1}^{m} y_j'(\pi) = 0,
\end{gather}
where $\textbf{q} = [q_1(x_1), \dots, q_m(x_m)]$ is a vector of real-valued potentials of class $L_2(0, \pi)$, $\la$ is the spectral parameter.
   

The theory of inverse spectral problems has been most thoroughly developed for Sturm--Liouville operators on intervals, as documented in several monographs \cite{Lev84, FY01, Mar11, Krav20} and related references. For differential operators on metric graphs -- often termed $\textit{quantum graphs}$, which model wave dynamics in spatial networks -- an extensive literature exists due to their relevance in classical and quantum mechanics, organic chemistry, nanotechnology, waveguide theory, and various engineering applications \cite{Berk, Pok}. Subsequent work established uniqueness results for recovering Sturm--Liouville potentials on arbitrary trees (graphs without cycles) from various spectral data \cite{Bel, Brow, YuGr}. The spectral mapping approach \cite{YuGr} yielded constructive reconstruction procedures for tree-graphs and, later, for arbitrary compact graphs \cite{YuCg}. Comprehensive reviews of inverse spectral theory for quantum graphs are provided in \cite{YuSn, BelR, BondPart} and recent monographs \cite{Piv, Kur}.

Although uniqueness and reconstruction aspects have been widely studied for inverse spectral problems on graphs, stability has received considerably less attention, with only fragmentary results available for specific graph structures. It is worth mentioning that the uniform stability of the inverse Sturm-Liouville problems on a finite interval has been proved in \cite{SavShk10, SavShk13}. Inverse Sturm--Liouville problems on compact star-shaped graphs were studied in matrix form with demonstrated stability in \cite{BondConst, BondMatr}. Stability results for inverse scattering on the half-line with general self-adjoint boundary conditions \cite{XuBond} apply directly to star-shaped graphs with infinite rays. Additionally, uniform stability was proved for a nonlocal functional-differential operator on a metric graph \cite{ButGr}, though such nonlocal operators differ fundamentally from local differential operators and require distinct analytical methods. Uniform stability was also proved in \cite{Bond25cycle, Trooshin} and \cite{BondTree} for lasso-shaped graph and tree, respectively.

To date, a general framework for analyzing stability of inverse problems for differential operators on graphs of arbitrary structure remains absent. Nevertheless, recent numerical methods for reconstructing Sturm--Liouville operators on graphs \cite{KravAvd, AvdIP, AvdMmas} show promising effectiveness in computational experiments, indirectly suggesting stability of the corresponding inverse problems and motivating further analytical investigation.

This paper aims to prove the uniform stability of two inverse spectral problems for \eqref{eqv1}--\eqref{kc}. In contrast to the previous studies \cite{Bond25cycle, Trooshin, BondTree} on the uniform stability, we use different types of spectral data. Namely, we consider the reconstruction of the potentials (i) from $m$ spectra, (ii) from eigenvalues and weight numbers. Our approach relies on the method of spectral mappings \cite{FY01, YuGr, BondConst, Bond25cycle} and on the uniform stability of recovering analytic functions from their zeros, which was obtained in \cite{But}.

The paper is organized as follows. Firstly, in Section~\ref{sec:main}, we introduce the notations and formulate the main theorems. Next, in Section~\ref{sec:stab}, we prove the theorem on the uniform stability for the case of given eigenvalue sets. Finally, in Section~\ref{sec:bound}, we prove the theorem on the uniform stability for the case of given eigenvalues and weight numbers. 

\section{Main results} \label{sec:main}

In this section, we give some notations and formulate the main theorems of our paper.

At first, let us define the characteristic function $\Delta(\la)$ of the problem $L$, which zeros coincide with the eigenvalues of the problem. In \cite{Pivovar} it is shown that
\begin{gather} \notag
    \Delta(\la) = \sum\limits_{k=1}^{m} S'_k(\pi, \la)\prod\limits_{j=1, j\neq k}^{m} S_j(\pi, \la),
\end{gather}
where $S_j(x, \la)$ is the solution of the $j$-th equation \eqref{eqv1} that satisfies the initial conditions $S_j(0, \la) = 0$, $S'_j(0, \la) = 1$.

Due to \cite{Pivovar}, the spectrum $\Lambda$ of the problem \eqref{eqv1}-\eqref{kc} is a countable set of eigenvalues  $\{\la_{nk}\}_{n \ge 1, \ k=\overline{1, m}}$. For simplicity we assume that $\int\limits_{0}^{\pi}q_j(x_j) \, dx_j = 0$ for any $k=\overline{1, m}$. Then, for an appropriate numbering of the eigenvalues $\{ \la_{nk}\}_{n \ge 1, \, k = \overline{1,m}}$ (counting with multiplicities), the following asymptotic relations hold:
\begin{gather} \label{eigen_asymp}
    \rho_{nk} = \sqrt{\la_{nk}} = n + \dfrac{\kappa_{nk}}{n\pi}, \\
    \label{eigen_asymp_m}
    \rho_{nm} = \sqrt{\la_{nm}} = n - \dfrac{1}{2} + \dfrac{\kappa_{nm}}{n\pi},
\end{gather}
where $n \ge 1$, $k = \overline{1, m-1}$, $\{ \kappa_{nk} \} \in l_2$.

For $j \in \{ 1, 2, \dots, m-1\}$, denote by $\Lambda_j$ the spectrum of the problem similar to \eqref{eqv1}--\eqref{kc}, but with changed boundary condition at the $j$-th vertex: $y'_j(0) = 0$. The spectra $\Lambda_j$ are countable sets of eigenvalues, which coincide with the zeros of the corresponding characteristic functions
$$
\Delta_j(\la) = \Bigg( C_j(x, \la) \prod\limits_{n=1, n\neq j}^{m} S_n(x, \la) \Bigg)' \Bigg|_{x=\pi},
$$
where $C_j(x, \la)$ be the solution of the $j$-th equation \eqref{eqv1} satisfying the initial conditions $C_j(0, \la) = 1$, $C'_j(0, \la) = 0$.
Note that the spectra $\Lambda_j$ are countable sets of eigenvalues $\{ \theta_{nkj}^2\}_{n \ge 1, k = \overline{1,m}}$ (counting with multiplicities), which can be numbered according their asymptotic behavior (see \cite{Cheng12}):
\begin{gather} \notag
    \theta_{nkj} = 
n + \dfrac{\xi_{nkj}}{n\pi}, \\
\notag
\theta_{n, m-1, j} =  n - 1 + \dfrac{1}{\pi} \arccos{\dfrac{1}{\sqrt{m}}} + \dfrac{\xi_{n,m-1,j}}{n\pi}, \\
\notag
\theta_{nmj} = n - \dfrac{1}{\pi} \arccos{\dfrac{1}{\sqrt{m}}} + \dfrac{\xi_{nmj}}{n\pi},
\end{gather}
where $n \ge 1$, $ k=\overline{1, m-2}$, $ j=\overline{1, m-1}$, $\{ \xi_{nkj} \} \in l_2$.

We study the following inverse spectral problem.

\begin{ip} \label{ip2}
    Given $\Lambda$, $\Lambda_j$, $j = \overline{1, m-1}$, find $q_j(x)$, $j=\overline{1, m}$.
\end{ip}

The uniqueness theorem and a constructive procedure for solving Inverse Problem~\ref{ip2} have been obtained in \cite{YuGr}. In this paper, we focus on the uniform stability of Inverse Problem~\ref{ip2}. 
Introduce two problems $L^{(1)} = L(\textbf{q}^{(1)})$ and $L^{(2)} = L(\textbf{q}^{(2)})$. We agree that, if an object $\gamma$ is related to the problem $L$, then $\gamma^{(\nu)}$ is the similar object related to the problem $L^{(\nu)}$ ($\nu = 1, 2$).

Introduce the set $\textbf{P}_Q := \left\{ \textbf{q}: \sum\limits_{i=1}^{m}\|q_j(x)\|_{L_2(0, \pi)} \le Q\right\}$, $Q > 0$. Below, we denote by $C(Q)$ various positive constants depending only on $Q$.

Our important result is the following theorem on uniform stability of Inverse Problem~\ref{ip2}:

\begin{thm} \label{thm_uni_stab_sp}
Let $Q > 0$ and $\textbf{q}^{(1)}, \textbf{q}^{(2)} \in \textbf{P}_Q$. Then
\begin{equation} \label{uni}
\sum\limits_{j=1}^{m} \| q_j^{(1)}(x) - q_j^{(2)}(x) \|_{L_2(0, \pi)} \le C(Q)\tilde\delta,
\end{equation}
where $\delta = \Bigg( \sum\limits_{n=1}^{\infty} \bigg(\sum\limits_{k=1}^{m}|n(\rho^{(1)}_{nk} - \rho^{(2)}_{nk})|^2 + \sum\limits_{k=1}^{m}\sum\limits_{j=1}^{m-1}|n(\theta_{nkj}^{(1)}-\theta_{nkj}^{(2)})|^2 \bigg) \Bigg)^{\frac{1}{2}}$.
\end{thm}

According to \cite{YuGr}, the Weyl solution $\Phi_{s}(x, \la)$ for each $s = \overline{1, m}$ is the vector $[\Phi_{js}(x, \la)]_{j=1}^{m}$ that satisfies equation \eqref{eqv1}, conditions \eqref{bc2}-\eqref{kc} and the boundary conditions:
$$
\Phi_{ss}(0, \la) = 1, \quad \Phi_{js}(0, \la) = 0, \quad j = \overline{1, m}\setminus s.
$$

The Weyl functions of the problem $L$ are defined as
$$
M_j(\la) = \Phi'_{jj}(0, \la).
$$

The Weyl functions are meromorphic in $\la$ and their poles coincide with the eigenvalues $\{ \la_{nk} \}$. Define the so-called weight numbers $\alpha_{nkj}$ as the following residues:
$$
\alpha_{nkj} = -\mathop{\mathrm{Res}}_{\la = \la_{nk}} M_j(\la).
$$

It follows from \cite[Theorem 3.1]{BondMMAS} that all $\alpha_{nkj} \ge 0$.
In the case of multiple eigenvalue $\lambda_{n_1 k_1} = \lambda_{n_2 k_2} = \dots = \lambda_{n_r k_r}$ we have $\alpha_{n_1 k_1 j} = \alpha_{n_2 k_2 j} = \dots = \alpha_{n_r k_r j}$. Let us split this weight number into the $r$ non-negative numbers: $\alpha_{n_1 k_1 j} = \sum_{i = 1}^r \beta_{n_i k_i j}$. 
Then, they satisfy the asymptotics (see \cite{Kuz18}):
\begin{gather*}
\sum_{k = 1}^{m-1} \beta_{nkj} =
\dfrac{2n^2}{m\pi}\Big( m-1+\dfrac{\varkappa_{nj}^0}{n} \Big), \\
\quad \beta_{nmj} = \dfrac{\Big( n - \dfrac{1}{2} \Big)^2}{m\pi}\Big( 2 + \dfrac{\varkappa_{nj}^1}{n} \Big),
\end{gather*}
where $n \ge 1$, $j = \overline{1,m}$ and $\{ \varkappa_{nj}^i \} \in l_2$, $i = 0, 1$.

Then, we can introduce the following inverse problem.

\begin{ip} \label{ip1}
    Given the spectral data $\{ \la_{nk}, \beta_{nkj} \}_{n \ge 1, k = \overline{1, m}, \ j = \overline{1,m}}$, find $q_j(x)$, $j = \overline{1, m}$.
\end{ip}

Our next important result is the following theorem on the uniform stability of Inverse Problem~\ref{ip1}:

\begin{thm} \label{thm_uni_stab}
Let $Q > 0$ and $\textbf{q}^{(1)}, \textbf{q}^{(2)} \in \textbf{P}_Q$. Then
\begin{equation*} 
\sum\limits_{j=1}^{m} \| q_j^{(1)}(x) - q_j^{(2)}(x) \|_{L_2(0, \pi)} \le C(Q)\tilde\delta,
\end{equation*}
where $\tilde\delta = \Bigg( \sum\limits_{n=1}^{\infty} \sum\limits_{k=1}^{m}(n\delta_{nk})^2 \Bigg)^{\frac{1}{2}}$, $\delta_{nk} = |\rho^{(1)}_{nk} - \rho^{(2)}_{nk}| + n^{-2}\sum\limits_{j=1}^{m}|\beta^{(1)}_{nkj} - \beta^{(2)}_{nkj}|$.

\end{thm}

Note that splitting of the weight numbers related to multiple eigenvalues should be chosen in the way to minimize the value $\delta$.

\section{Uniform stability of the inverse problem by the set of spectra} \label{sec:stab}

In this section, we prove Theorem~\ref{thm_uni_stab_sp} on the uniform stability of the inverse problem by the set of spectra. At first, we get the estimates for the kernel functions of characteristic function and some additional functions. Then, using the Fourier coefficients we get the required estimate for the potentials.

\begin{prop}[\cite{FY01}] \label{prop:KN}
    Let $\mathbf{q} \in \mathbf{P}_Q$. Then, the following relations hold for $j = \overline{1, m}$:
    \begin{gather}
        \label{s_asymp}
        S_j(x, \la) = \dfrac{\sin\rho x}{\rho} + \dfrac{1}{\rho^2} \int\limits_0^x {K_j(x, t)}\cos\rho t \, dt, \\
        \notag
        S'_j(x, \la) = \cos\rho x + \dfrac{1}{\rho} \int\limits_0^x {N_j(x, t)}\sin\rho t \, dt,
    \end{gather}
    where $K_j(x, \cdot)$, $N_j(x, \cdot)$ $\in L_2(0, x)$. Moreover, $\| K_j(x, \cdot) \|_{L_2(0, x)} \le C(Q)$, $\| N_j(x, \cdot) \|_{L_2(0, x)} \le C(Q)$ for $j = \overline{1,m}$, $x \in (0, T_j]$.
\end{prop}

Recall, that $\Delta(\la)$ is the characteristic function of the problem \eqref{eqv1}-\eqref{kc}, and $\Delta_j(\la)$ is the characteristic function of the problem with the same form, but with changed boundary condition on the $j$-th vertex: $y'_j(0) = 0$.

For $T > 0$, denote by $PW(T)$ the class of the Paley-Wiener functions of the form
$$
\mathcal F(\rho) = \int_{-T}^T f(t) e^{i \rho t} \, dt, \quad f \in L_2(-T, T).
$$

\begin{lem}[Corollary 3.3 from \cite{Bond25cycle}]
    The following representations hold
    \begin{gather} \notag
        \Delta(\la) = \dfrac{m\sin^{m-1}\rho\pi\cos\rho\pi}{\rho^{m-1}} + \dfrac{F(\rho)}{\rho^m}, \\ \notag
        \Delta_k(\la) = \dfrac{\sin^{m-2}\rho\pi(m\cos^2\rho\pi - 1)}{\rho^{m-2}} + \dfrac{F_k(\rho)}{\rho^{m-1}},
    \end{gather}
    where $F$, $F_k \in PW(m\pi)$, $k = \overline{1, m-1}$.
\end{lem}

Introduce two problems $L^{(1)}$ and $L^{(2)}$ such that $\textbf{q}^{(j)} \in \textbf{P}_Q$, $j=1, 2$.

\begin{lem} \label{lem:difF}
    The following estimates hold:
    $$
    \| F^{(1)}(\rho) - F^{(2)}(\rho) \|_{L_2(\mathbb{R})} \le C(Q)\delta, \quad \| F_k^{(1)}(\rho) - F_k^{(2)}(\rho) \|_{L_2(\mathbb{R})} \le C(Q)\delta, \quad k = \overline{1,m-1}.
    $$
\end{lem}

\begin{proof}
    Split the full spectrum of $\Delta(\la)$ into $m$ subsequences $\{ \la_{nk} \}_{n\ge 1}$, $k = \overline{1, m}$.
    For every subsequence, we can build a function $\theta_k(\la)$ of the following form (see \cite[Remark 1.1.2]{FY01}):
    $$
        \theta_k(\la) = \begin{cases} 
            \dfrac{\sin\rho\pi}{\rho} + \dfrac{f_k(\rho)}{\rho^2} = \pi\displaystyle\prod\limits_{n=1}^{\infty} \dfrac{\la_{nk} - \la}{n^2}, & k = \overline{1, m-1}, \\
            \cos\rho\pi + \dfrac{f_m(\rho)}{\rho} = \displaystyle\prod\limits_{n=1}^{\infty} \dfrac{\la_{nm} - \la}{(n-\frac{1}{2})^2}, & k = m,
        \end{cases} 
    $$
    where $f_k \in PW(\pi)$, $k = \overline{1,m}$.
    
    Next, for each function $\rho\theta_k(\la)$ we can apply the Theorem~7 from \cite{But} and get
    $$
    \| f^{(1)}_k - f^{(2)}_k \|_{L_2(\mathbb R)} \le C_r\| \{ n(\rho^{(1)}_{nk} - \rho^{(2)}_{nk}) \}_{n \ge 1} \|_{l_2},
    $$
    where the remainder terms of the asymptotics \eqref{eigen_asymp} and \eqref{eigen_asymp_m} satisfy $\|\{ \kappa_{nk}^{(i)} \}\|_{l_2} \le r$, $i = 1, 2$, and the constant $C_r$ depends only on $r$.
    
    Next, let us pass to the function $\Delta(\la)$ as follows:
    $$
    \Delta(\la) = \pi^{m-1}\prod\limits_{n=1}^{\infty}\dfrac{\la_{nm} - \la}{(n-\frac{1}{2})^2}\prod\limits_{k=1}^{m-1} \dfrac{\la_{nk} - \la}{n^2} = \prod\limits_{k=1}^{m}\theta_k(\la).
    $$

    Finally, pass to the difference:
    $$
    \Delta^{(1)}(\la) - \Delta^{(2)}(\la) = \sum\limits_{j=1}^{m}\prod\limits_{k=1}^{j-1}\theta^{(1)}_k(\la)(\theta^{(1)}_j(\la) - \theta^{(2)}_j(\la))\prod\limits_{k=j+1}^{m}\theta^{(2)}_j(\la),
    $$
    then
        $$
    \| F^{(1)}(\rho) - F^{(2)}(\rho) \|_{L_2(\mathbb{R})} \le C_r\delta.
    $$
    As all $\{ \kappa_{nk} \}_{n\ge 1}$, $k = \overline{1, m}$, depend on potentials $\textbf{q}$, which are bounded by constant $Q$, we can say, that $C_r = C(Q)$. This yields the claim for $F^{(1)}(\rho) - F^{(2)}(\rho)$. The estimate for $F_k^{(1)}(\rho) - F_k^{(2)}(\rho)$ is proved analogously.
\end{proof}

At first, we need to prove the required estimate for potentials of the vertexes with known spectra. In \cite{Bond25cycle} analogous theorem is proved for a graph with a cycle. But, assuming the length of the cycled edge equal to zero, we pass to the star-shaped graph without cycle with the number of edges $1$ less than the original one.

\begin{lem} [Theorem 2.3 from \cite{Bond25cycle}] \label{lem:bound}
    Let $Q > 0$ be fixed. Then, for any $\textbf{q}^{(1)}, \textbf{q}^{(2)} \in \textbf{P}_Q$, there holds
    $$
    \| q^{(1)}_k - q^{(2)}_k \|_{L_2(0, \pi)} \le C(Q) \bigg( \| \rho^{m}(\Delta^{(1)}-\Delta^{(2)})(\rho^2) \|_{L_2(\mathbb{R})} + \| \rho^{m-1}(\Delta^{(1)}_k - \Delta^{(2)}_k)(\rho^2) \|_{L_2(\mathbb{R})} \bigg),
    $$
    for $k = \overline{1, m-1}$.
\end{lem}

Combining Lemmas~\ref{lem:difF} and \ref{lem:bound}, we arrive at the estimates \eqref{uni} for $k = \overline{1,m-1}$. It remains to prove that estimate for $k = m$.

Introduce the following functions:
\begin{gather}
    \notag
    \Delta^\Pi(\la) = \prod\limits_{n=1}^{m-1}S_n(\pi, \la), \\
    \notag
    \Delta^\Pi_1(\la) = C_1(\pi, \la)\prod\limits_{n=2}^{m-1}S_n(\pi, \la), \\
    \notag
    \Delta^K(\la) = \sum\limits_{j=1}^{m-1} S'_j(\pi, \la)\prod\limits_{n=1, n \neq j}^{m-1}S_n(\pi, \la), \\
    \notag
    \Delta^K(\la) = C'_1(\pi, \la)\prod\limits_{n=2}^{m-1}S_n(\pi, \la) + \sum\limits_{j=2}^{m-1} C_1(\pi, \la) S'_j(\pi, \la)\prod\limits_{n=2, n \neq j}^{m-1}S_n(\pi, \la).
\end{gather}

Denote $p:=(m-1)\pi$.

\begin{lem}[Corollary 3.3 from \cite{Bond25cycle}]
    The following relations hold
    \begin{gather}
    \notag
        \Delta^\Pi(\la) = \dfrac{\sin^{m-1}\rho\pi}{\rho^{m-1}} + \dfrac{1}{\rho^{m}}\int\limits_{-p}^{p}f(t)e^{-i\rho t} \, dt, \\
        \label{delta_Pi_1}
        \Delta^\Pi_1(\la) = \dfrac{\sin^{m-2}\rho\pi \cos\rho\pi}{\rho^{m-2}} + \dfrac{1}{\rho^{m-1}}\int\limits_{-p}^{p}f_1(t)e^{-i\rho t} \, dt, \\ \notag
        \Delta^K(\la) = \dfrac{(m-1)\sin^{m-2}\rho\pi \cos\rho\pi}{\rho^{m-2}} + \dfrac{1}{\rho^{m-1}}\int\limits_{-p}^{p}g(t)e^{-i\rho t} \, dt, \\ \notag
        \Delta^K_1(\la) = -\dfrac{\sin^{m-1}\rho\pi}{\rho^{m-3}} + \dfrac{(m-2)\sin^{m-3}\rho\pi \cos^2\rho\pi}{\rho^{m-3}} + \dfrac{1}{\rho^{m-2}}\int\limits_{-p}^{p}g_1(t)e^{-i\rho t} \, dt,
    \end{gather}
    where $f(t), \, f_1(t), \, g(t), \, g_1(t) \in L_2(-p, p)$.
\end{lem}

For the functions $\rho^m\Delta^\Pi(\la)$, $\rho^{m-1}\Delta_1^\Pi(\la)$, $\rho^{m-1}\Delta^K(\la)$, $\rho^{m-2}\Delta_1^K(\la)$, we can get the following lemma.

\begin{lem}[Corollary 3.5 from \cite{Bond25cycle}] \label{lem:difD}
The following estimates hold:
\begin{gather*}
    \| f^{(1)}(t) - f^{(2)}(t) \|_{L_2(-p, p)} \le C(Q)\sum\limits_{n=1}^{m-1}\|q^{(1)}_n - q^{(2)}_n\|_{L_2(0, \pi)}, \\
    \| f_1^{(1)}(t) - f_1^{(2)}(t) \|_{L_2(-p, p)} \le C(Q)\sum\limits_{n=1}^{m-1}\|q^{(1)}_n - q^{(2)}_n\|_{L_2(0, \pi)}, \\
    \| g^{(1)}(t) - g^{(2)}(t) \|_{L_2(-p, p)} \le C(Q)\sum\limits_{n=1}^{m-1}\|q^{(1)}_n - q^{(2)}_n\|_{L_2(0, \pi)}, \\
    \| g_1^{(1)}(t) - g_1^{(2)}(t) \|_{L_2(-p, p)} \le C(Q)\sum\limits_{n=1}^{m-1}\|q^{(1)}_n - q^{(2)}_n\|_{L_2(0, \pi)}.    
\end{gather*}
\end{lem}

Combining the estimates of Lemma~\ref{lem:difD} with \eqref{uni} for $k = \overline{1,m-1}$, which is already proved, we obtain
\begin{gather}
\nonumber 
\| f^{(1)}(t) - f^{(2)}(t) \|_{L_2(-p, p)} \le C(Q)\delta, \\
\label{f1_est}
\| f_1^{(1)}(t) - f_1^{(2)}(t) \|_{L_2(-p, p)} \le C(Q)\delta, \\
\nonumber 
\| g^{(1)}(t) - g^{(2)}(t) \|_{L_2(-p, p)} \le C(Q)\delta, \\
\nonumber 
\| g_1^{(1)}(t) - g_1^{(2)}(t) \|_{L_2(-p, p)} \le C(Q)\delta.    
\end{gather}

Note that
\begin{gather*}
    \Delta(\la) = \Delta^K(\la)S_m(\pi, \la) + \Delta^\Pi(\la)S'_m(\pi, \la), \\
    \Delta_1(\la) = \Delta_1^K(\la)S_m(\pi, \la) + \Delta_1^\Pi(\la)S'_m(\pi, \la).
\end{gather*}
Then, using simple calculations, we can get that
\begin{gather} \label{Sm}
    S_m(\pi, \la) = \dfrac{\Delta(\la)\Delta_1^\Pi(\la) - \Delta_1(\la)\Delta^\Pi(\la)}{B(\la)}, \\ \nonumber
    S'_m(\pi, \la) = \dfrac{\Delta_1(\la)\Delta^K(\la) - \Delta(\la)\Delta_1^K(\la)}{B(\la)} = -S_m(\pi, \la)\sum\limits_{j=2}^{m-1}\dfrac{S'_j(\pi, \la)}{S_j(\pi, \la)},
\end{gather}
where $B(\la) = \prod\limits_{k=2}^{m-1}S^2_k(\pi, \la)$.

Let us pass from the functions $S_m(\pi, \la)$ and $S'_m(\pi, \la)$ to the kernels $K_m(\pi, t)$ and $N_m(\pi, t)$ defined in Proposition~\ref{prop:KN}. For this purpose, fix numbers $\nu_n = n + i\alpha$, $\alpha > 0$, $n \in \mathbb{Z}$, and define $\mu_n = \nu_n^2$. Then $\{ e^{-i\nu_n t}\}_{n \in \mathbb Z}$ is a Riesz basis in $L_2(-\pi, \pi)$.
Introduce the Fourier coefficients
\begin{gather*}
\hat k_n = \int\limits_{-\pi}^{\pi} \mathcal{K}(t)e^{-i\nu_nt} \, dt = \mu_nS_m(\pi, \mu_n)-\nu_n\sin\nu_n\pi, \quad \mathcal{K}(t) = \left\{ \begin{aligned} 
\dfrac{1}{2}K_m(\pi, t), \ t \ge 0, \\
\dfrac{1}{2}K_m(\pi, -t), \ t < 0,
\end{aligned} \right.
\end{gather*}

Due to the Riesz-basis property, there holds
\begin{equation} \label{Riesz}
c \| \{ \hat k_n\}_{n \in \mathbb Z}\|_{l_2} \le \| K_m(\pi, t) \|_{L_2(0,\pi)} \le C \| \{ \hat k_n \}_{n \in \mathbb Z} \|_{l_2},
\end{equation}
where positive constants $c$ and $C$ depend only on $\alpha$.

Analogously, define
\begin{gather*}
\hat h_n = \int\limits_{-\pi}^{\pi} \mathcal{N}(t)e^{-i\nu_nt} \, dt = \nu_n(S'_m(\pi, \mu_n)-\cos\nu_n\pi), \quad \mathcal{N}(t) = \left\{ \begin{aligned} 
-\dfrac{1}{2i}N_m(\pi, t), \ t \ge 0, \\
\dfrac{1}{2i}N_m(\pi, -t), \ t < 0.
\end{aligned} \right.
\end{gather*}

\begin{lem}
	The following relations hold for $n \in \mathbb Z$:
	\begin{gather} \label{difkn}
	\| \{ \hat k^{(1)}_n - \hat k^{(2)}_n \} \|_{l_2} \le C(Q)\delta, \\ \label{difhn}
	\| \{ \hat h^{(1)}_n - \hat h^{(2)}_n \} \|_{l_2} \le C(Q)\delta.
	\end{gather}
\end{lem}

\begin{proof}
Consider the difference
$$
\hat k^{(1)}_n - \hat k^{(2)}_n = \mu_n(S_m^{(1)}(\pi, \mu_n) - S_m^{(2)}(\pi, \mu_n)).
$$
	
Denote $J(\mu_n) = \Delta(\mu_n)\Delta^{\Pi}_1(\mu_n) - \Delta_1(\mu_n)\Delta^{\Pi}(\mu_n)$ and $\check\gamma = \gamma^{(1)} - \gamma^{(2)}$. Using \eqref{Sm}, we obtain
\begin{equation} \label{relkn}
\hat k^{(1)}_n - \hat k^{(2)}_n = \dfrac{\mu_n\check B(\mu_n)J(\mu_n)}{B^{(1)}(\mu_n)B^{(2)}(\mu_n)} + \dfrac{ \mu_n\check J(\mu_n)}{B^{(2)}(\mu_n)}.
\end{equation}

Let us get some estimates of the elements of this difference.

\begin{enumerate}
	\item $\big| B^{(2)}(\mu_n) \big|$.\\
	As $|\sin\nu_n\pi| \ge C(Q)$, then $\big| S^{(2)}_k(\mu_n, \pi) \big| \ge C(Q)|\nu_n|^{-1}$. Using this estimate we pass to $|B^{(2)}(\mu_n)| \ge C(Q)|\nu_n|^{4-2m}$.
	\item $\| \{ \nu_n^{2m-3} \check B(\mu_n) \} \|_{l_2}$.\\
	The value $\check B(\mu_n)$ can be presented in the form:
	\begin{multline*}
	\check B(\mu_n) = \sum\limits_{j=2}^{m-1}\prod\limits_{k=2}^{j-1}\bigl( S^{(1)}_k(\mu_n, \pi)\bigr)^2 \bigl(S^{(1)}_j(\mu_n, \pi) - S^{(2)}_j(\mu_n, \pi)\bigr) \\ \times \bigl(S^{(1)}_j(\mu_n, \pi) + S^{(2)}_j(\mu_n, \pi)\bigr)\prod\limits_{k=j+1}^{m-1}\bigl( S^{(2)}_k(\mu_n, \pi)\bigr)^2.
	\end{multline*}
	Then, using \eqref{s_asymp} we can get:
	\begin{gather*}
	|\check B(\mu_n)| \le |\nu_n|^{3-2m} \sum_{j = 2}^{m-1}\Bigg| \int\limits_{0}^{\pi} {\Big(K^{(1)}_j(\pi, t) - K^{(2)}_j(\pi, t) \Big) \cos\nu_n t \, dt} \Bigg|, \quad n \in \mathbb Z.
	\end{gather*}
	Note that
	\begin{align*}
		\left\| \Biggl\{ \int\limits_{0}^{\pi} {\Big(K^{(1)}_j(\pi, t) - K^{(2)}_j(\pi, t) \Big) \cos\nu_n t \, dt} \Biggr\}_{n \in \mathbb Z}\right\|_{l_2} & \le C \| K_j^{(1)} - K_j^{(2)} \|_{L_2(0,\pi)} \\ & \le C(Q) \| q_j^{(1)} - q_j^{(2)} \|_{L_2(0,\pi)}, \quad j = \overline{2,m-1},
	\end{align*}
	according to formula (22) in \cite{But21}. Taking the estimates \eqref{uni} for $j = \overline{2,m-1}$ into account, we get
	$\| \{ \nu_n^{2m-3} \check B(\mu_n) \} \|_{l_2} \le C(Q)\delta$.
	\item $|\Delta^{\Pi(2)}_1(\nu_n^2)|$.
	The estimate $|\Delta^{\Pi(2)}_1(\nu_n^2)| \le C(Q)|\nu_n|^{1-m}$ follows from \eqref{delta_Pi_1}.
	\item $\| \{ \nu_n^{m-1}\check\Delta^\Pi_1(\nu_n^2) \} \|_{l_2}$
	Using \eqref{delta_Pi_1} we can get:
	\begin{gather*}
	\Delta^{\Pi (1)}_1(\mu_n) - \Delta^{\Pi (2)}_1(\mu_n) = \dfrac{1}{\rho^{m-1}}\int\limits_{-p}^{p} {(f_1^{(1)}(t)-f_1^{(2)}(t))e^{-i\nu_nt}} \, dt,
	\end{gather*}
	and, consequently, the estimate \ref{f1_est} together with the Riesz-basis property of $\{ e^{-i \nu_n t}\}_{n \in \mathbb Z}$ imply
	$$
	\| \{ \nu_n^{m-1}\check \Delta^\Pi_1(\nu_n^2) \} \|_{l_2} \le C(Q)\delta.
	$$
\end{enumerate}

We can get other estimates $|\Delta^{\Pi(2)}(\nu_n^2)|\le C(Q)|\nu_n|^{-m}$, $\| \{ \nu_n^{m}\check \Delta^\Pi(\nu_n^2) \} \|_{l_2} \le C(Q)\delta$, $|\Delta^{(1)}(\nu_n^2)|\le C(Q)|\nu_n|^{-m}$, $\| \{ \nu_n^{m}\check \Delta(\nu_n^2) \} \|_{l_2} \le C(Q)\delta$, $|\Delta^{(1)}_1(\nu_n^2)|\le C(Q)|\nu_n|^{1-m}$, $\| \{ \nu_n^{m-1}\check \Delta_1(\nu_n^2) \} \|_{l_2} \le C(Q)\delta$ in the same way. Substituting these estimates into \eqref{relkn}, we arrive at \eqref{difkn}. The estimate \eqref{difhn} is proved analogously.
\end{proof}

Now, we are ready to finish the proof of Theorem~\ref{thm_uni_stab_sp}.
The relations \eqref{Riesz} and \eqref{relkn} together imply
\begin{equation} \label{estK}
\| K^{(1)}_m(\pi, t) - K^{(2)}_m(\pi, t) \|_{L_2(0, \pi)} \le C(Q)\delta.
\end{equation}
In the same way we can get that
\begin{equation} \label{estN}
\| N^{(1)}_m(\pi, t) - N^{(2)}_m(\pi, t) \|_{L_2(0, \pi)} \le C(Q)\delta.
\end{equation}

Consider the inverse problem of recovering the potential $q_m$ from the so-called Cauchy data $\{ K_m(\pi, t), N_m(\pi, t)\}$. Obviously, this problem is equivalent to the classical Borg inverse problem by two spectra (see, e.g., \cite{FY01}). The uniform stability of the potential reconstruction from the Cauchy data has been established in \cite{Bond25} (the case of the Dirichlet boundary conditions has minor technical differences, see \cite{Bond20}):
$$
\| q_m^{(1)} - q_m^{(2)} \|_{L_2(0,\pi)} \le C(Q) \bigl(\| K^{(1)}_m(\pi, t) - K^{(2)}_m(\pi, t) \|_{L_2(0, \pi)} + \| N^{(1)}_m(\pi, t) - N^{(2)}_m(\pi, t) \|_{L_2(0, \pi)}\bigr). 
$$

Combining the latter estimate with \eqref{estK} and \eqref{estN}, we arrive at \eqref{uni} for $j = m$, which concludes the proof of Theorem~\ref{thm_uni_stab_sp}.

\section{Uniform stability of the inverse problem by spectral data} \label{sec:bound}

In this section, we provide the proof of Theorem~\ref{thm_uni_stab} on the uniform stability of the inverse problem by spectral data. 

Using the method of spectral mappings \cite{FY01, YuGr, BondConst}, we get the following expression (see (10.24) in \cite{BondMatr}):
\begin{gather}
    \label{q_rec}
    q^{(1)}_j(x) - q^{(2)}_j(x) =  2 \sum\limits_{n=1}^{\infty}\sum\limits_{k=1}^{m}\sum\limits_{i=1}^{2}(-1)^i \beta_{nkj}^{(i)} \frac{d}{dx}\bigl(S^{(1)}_j(x, \la_{nk}^{(i)})S^{(2)}_j(x, \la_{nk}^{(i)})\bigr).
\end{gather}

\begin{proof}[Proof of Theorem~\ref{thm_uni_stab}]
    From \eqref{q_rec}, we get that $q^{(1)}_j(x) - q^{(2)}_j(x) = \sum\limits_{k=1}^{6} S_{kj}(x)$,
    where
    \begin{gather}
        \label{ser1}
        S_{1j}(x) = \sum\limits_{n=1}^{\infty}\sum\limits_{k=1}^{m}(\beta^{(2)}_{nkj} - \beta^{(1)}_{nkj})\frac{d}{dx}{S^{(2)}_j}(x, \la^{(2)}_{nk}){S^{(1)}_j}(x, \la^{(2)}_{nk}), \\
        \label{ser2}
        S_{2j}(x) = \sum\limits_{n=1}^{\infty}\sum\limits_{k=1}^{m} \beta^{(1)}_{nkj}\frac{d}{dx}({S^{(2)}_j}(x, \la^{(2)}_{nk}) - {S^{(2)}_j}(x, \la^{(1)}_{nk})){S^{(1)}_j}(x, \la^{(2)}_{nk}), \\
        \label{ser3}
        S_{3j}(x) = \sum\limits_{n=1}^{\infty}\sum\limits_{k=1}^{m} \beta^{(1)}_{nkj}\frac{d}{dx}{S^{(2)}_j}(x, \la^{(1)}_{nk})({S^{(1)}_j}(x, \la^{(2)}_{nk}) - {S^{(1)}_j}(x, \la^{(1)}_{nk})), \\
        \label{ser4}
        S_{4j}(x) = \sum\limits_{n=1}^{\infty} \sum\limits_{k=1}^{m} (\beta^{(2)}_{nkj} - \beta^{(1)}_{nkj}){S^{(2)}_j}(x, \la^{(2)}_{nk})\frac{d}{dx}{S^{(1)}_j}(x, \la^{(2)}_{nk}), \\
        \label{ser5}
        S_{5j}(x) = \sum\limits_{n=1}^{\infty} \sum\limits_{k=1}^{m} \beta^{(1)}_{nkj}({S^{(2)}_j}(x, \la^{(2)}_{nk}) - {S^{(2)}_j}(x, \la^{(1)}_{nk}))\frac{d}{dx}{S^{(1)}_j}(x, \la^{(2)}_{nk}), \\
        \label{ser6}
        S_{6j}(x) = \sum\limits_{n=1}^{\infty} \sum\limits_{k=1}^{m} \beta^{(1)}_{nkj}{S^{(2)}_j}(x, \la^{(1)}_{nk})\frac{d}{dx}({S^{(1)}_j}(x, \la^{(2)}_{nk}) - {S^{(1)}_j}(x, \la^{(1)}_{nk})).
    \end{gather}

    Let us start from the series \eqref{ser1}:
    \begin{align}
        \notag
        S_{1j}(x) &= \sum\limits_{n=1}^{\infty} \left( \sum\limits_{k=1}^{m-1}\dfrac{\beta_{nkj}^{(2)} - \beta_{nkj}^{(1)}}{n} \cos nx \sin nx + \dfrac{\beta_{nmj}^{(2)} - \beta_{nmj}^{(1)}}{n-\frac{1}{2}} \cos \left(n-\dfrac{1}{2}\right)x \sin \left(n-\dfrac{1}{2}\right)x\right) \\ \notag & +\sum\limits_{n=1}^{\infty}\sum\limits_{k=1}^{m}O(\delta_{nk}).
    \end{align}

	Since $|\beta_{nkj}^{(2)} - \beta_{nkj}^{1}| \le n^2 \delta_{nk}$ and $\{ n \delta_{nk}\} \in l_2$ if $\tilde \delta < \infty$, then the series \eqref{ser1} converges in $L_2(0, \pi)$. Moreover, we have $\| N^{(2)}_j(x, .) \|_{L_2(0, x)} \le C(Q)$ and $\| K^{(1)}_j(x, .) \|_{L_2(0, x)} \le C(Q)$ by Proposition~\ref{prop:KN}, so
    $$
    \Bigg\| \sum\limits_{j=1}^{m} S_{1j}(x) \Bigg\|_{L_2(0, \pi)} \le C(Q)\tilde\delta.
    $$

    In the same way we get the similar estimate for $S_{4j}$. 
    Analogously, we obtain the estimates
    $$
    \Bigg| \sum_{j=1}^{m} S_{kj}(x) \Bigg| \le C(Q)\tilde\delta, \quad k=2, 3, 5, 6.
    $$
    Convergence and the estimates for \eqref{ser1}--\eqref{ser6} yield the claim.
\end{proof}

\medskip

\textbf{Funding}: This work was supported by Grant 24-71-10003 of the Russian Science Foundation, https://rscf.ru/en/project/24-71-10003/.

\medskip







\medskip

\noindent Egor Evgenevich Chitorkin \\
1. Institute of IT and Cybernetics, Samara National Research University,\\ 
Moskovskoye Shosse 34, Samara 443086, Russia. \\
2. Department of Mechanics and Mathematics, Saratov State University, \\
Astrakhanskaya 83, Saratov 410012, Russia. \\
e-mail: {\it chitorkin.ee@ssau.ru}

\medskip

\noindent Natalia Pavlovna Bondarenko \\
1. Department of Applied Mathematics, Samara National Research University,\\
Moskovskoye Shosse 34, Samara 443086, Russia.\\
2. Department of Mechanics and Mathematics, Saratov State University, \\
Astrakhanskaya 83, Saratov 410012, Russia. \\
3. S.M. Nikolskii Mathematical Institute, Peoples' Friendship University of Russia (RUDN University), 6 Miklukho-Maklaya Street, Moscow, 117198, Russia,\\
4. Moscow Center of Fundamental and Applied Mathematics, Lomonosov Moscow State University, Moscow 119991, Russia.\\
e-mail: {\it bondarenkonp@sgu.ru}\\


\begin{thebibliography}{99}



\bibitem{Lev84}
Levitan, B.~M. \emph{Inverse Sturm-Liouville Problems}. VNU Sci. Press, Utrecht, 1987.

\bibitem{FY01}
Freiling, G.; Yurko, V.~A. \emph{Inverse Sturm-Liouville Problems and Their Applications}. Nova Science Publishers, Huntington, NY, 2001.

\bibitem{Mar11}
Marchenko, V.~A. \emph{Sturm-Liouville Operators and Applications}. Revised edition, AMS, Providence, 2011.

\bibitem{Krav20}
Kravchenko, V.~V. \emph{Direct and Inverse Sturm-Liouville Problems}. Birkh\"auser, Cham, 2020.

\bibitem{Pok}
Pokorny, Yu.~V.; Penkin, O.~M.; Pryadiev, V.~L. et al. \emph{Differential Equations on Geometrical Graphs}. Fizmatlit, Moscow, 2005 [in Russian].

\bibitem{Berk}
Berkolaiko, G.; Carlson, R.; Fulling, S.; Kuchment, P. \emph{Quantum Graphs and Their Applications}. Contemp. Math. 415, Amer. Math. Soc., Providence, RI, 2006.

\bibitem{Bel}
Belishev, M.~I. Boundary spectral inverse problem on a class of graphs (trees) by the BC-method. \emph{Inverse Probl.} \textbf{20} (2004), 647--672.

\bibitem{Brow}
Brown, B.~M.; Weikard, R. A Borg-Levinson theorem for trees. \emph{Proc. Royal Soc. A: Math. Phys. Engin. Sci.} \textbf{461} (2005), 3231--3243.

\bibitem{YuGr}
Yurko, V.~A. Inverse spectral problems for Sturm-Liouville operators on graphs. \emph{Inverse Probl.} \textbf{21} (2005), 1075--1086.

\bibitem{YuCg}
Yurko, V.~A. Inverse spectral problems for differential operators on arbitrary compact graphs. \emph{J. Inverse Ill-Posed Probl.} \textbf{18} (2010), no.~3, 245--261.

\bibitem{YuSn}
Yurko, V.~A. Inverse spectral problems for differential operators on spatial networks. \emph{Russ. Math. Surveys} \textbf{71} (2016), no.~3, 539--584.

\bibitem{BelR}
Belishev, M.~I. Boundary control and tomography of Riemannian manifolds (the BC-method). \emph{Russ. Math. Surveys} \textbf{72} (2017), no.~4, 581--644.

\bibitem{BondPart}
Bondarenko, N.~P. Partial inverse Sturm-Liouville problems. \emph{Mathematics} \textbf{11} (2023), no.~10, 2408.

\bibitem{Piv}
Moller, M.; Pivovarchik, V. \emph{Direct and Inverse Finite-Dimensional Spectral Problems on Graphs}. Birkh\"auser, Cham, 2020.

\bibitem{Kur}
Kurasov, P. \emph{Spectral Geometry of Graphs}. Birkh\"auser, Berlin, 2024.

\bibitem{SavShk10}
Savchuk, A.~M.; Shkalikov, A.~A. Inverse problems for Sturm-Liouville operators with potentials in Sobolev spaces: Uniform stability, \emph{Funct. Anal. Appl.} 44 (2010), no. 4, 270--285.

\bibitem{SavShk13}
Savchuk, A.~M.; Shkalikov, A.~A. Uniform stability of the inverse Sturm-Liouville problem with respect to the spectral function in the scale of Sobolev spaces. \emph{Proc. Steklov Inst. Math.} \textbf{283} (2013), 181--196. \url{https://doi.org/10.1134/S0081543813080130}.

\bibitem{BondConst}
Bondarenko, N.~P. Constructive solution of the inverse spectral problem for the matrix Sturm-Liouville operator. \emph{Inverse Probl. Sci. Engin.} \textbf{28} (2020), no.~9, 1307--1330.

\bibitem{BondMatr}
Bondarenko, N.~P. Uniform stability for the matrix inverse Sturm-Liouville problems. arXiv:2506.15300.

\bibitem{XuBond}
Xu, X.~C.; Bondarenko, N.~P. Stability of the inverse scattering problem for the self-adjoint matrix Schr\"odinger operator on the half line. \emph{Stud. Appl. Math.} \textbf{149} (2022), no.~3, 815--838.

\bibitem{ButGr}
Buterin, S.~A. Functional-differential operators on geometrical graphs with global delay and inverse spectral problems. \emph{Results Math.} \textbf{78} (2023), 79.

\bibitem{Bond25cycle}
Bondarenko, N.~P. Stability of the inverse Sturm-Liouville problem on a graph with a cycle. \emph{J. Inverse Ill-Posed Probl.} (2025), published online. URL: \url{https://doi.org/10.1515/jiip-2025-0059}.

\bibitem{Trooshin}
Mochizuki, K.; Trooshin, I. On Conditional Stability of Inverse Scattering Problem on a Lasso-Shaped Graph. In: \emph{Analysis, Probability, Applications, and Computation}. Trends in Mathematics, Birkh\"auser, Cham, 2019. \url{https://doi.org/10.1007/978-3-030-04459-6\_19}.

\bibitem{BondTree}
Bondarenko, N.~P. Stability of the inverse Sturm-Liouville problem on a quantum tree, Stud. Appl. Math. 155 (2025), no. 6, Article ID e70162.

\bibitem{KravAvd}
Kravchenko, V.~V.; Avdonin, S.~A. Method for solving inverse spectral problems on quantum star graphs. \emph{J. Inv. Ill-Posed Probl.} \textbf{31} (2023), no.~1, 31--42.

\bibitem{AvdIP}
Avdonin, S.~A.; Khmelnytskaya, K.~V.; Kravchenko, V.~V. Recovery of a potential on a quantum star graph from Weyl's matrix. \emph{Inverse Probl. Imag.} \textbf{18} (2024), no.~1, 311--325.

\bibitem{AvdMmas}
Avdonin, S.~A.; Khmelnytskaya, K.~V.; Kravchenko, V.~V. Reconstruction techniques for quantum trees. \emph{Math. Meth. Appl. Sci.} \textbf{47} (2024), no.~9, 7182--7197.

\bibitem{But}
Buterin, S.~A. On the uniform stability of the recovery of sine-type functions with asymptotically separated zeros. \emph{Math. Notes} \textbf{111} (2022), no.~3, 339--353.

\bibitem{Pivovar}
Pivovarchik, V. Inverse problem for the Sturm--Liouville equation on a star-shaped graph. \emph{Math. Nachr.} \textbf{280} (2007), no.~13--14, 1595--1619.

\bibitem{Cheng12}
Cheng, Y.~H. Reconstruction of the Sturm-Liouville operator on a p-star graph with nodal data. \emph{Rocky Mount. J. Math.} \textbf{42} (2012), no.~5, 1431--1446.

\bibitem{BondMMAS}
Bondarenko, N.~P. Spectral analysis of the Sturm-Liouville operator on the star-shaped graph. \emph{Math. Meth. Appl. Sci.} \textbf{43} (2020), 471--485. \url{https://doi.org/10.1002/mma.5853}.

\bibitem{Kuz18}
Kuznetsova, M.~A. Asymptotic formulae for weight numbers of the Sturm--Liouville boundary problem on a star-shaped graph. \emph{Proc. of Saratov University: Mathematics. Mechanics. Informatics} \textbf{18} (2018), no.~1, 40--48.

\bibitem{But21}
Buterin, S. Uniform full stability of recovering convolutional perturbation of the Sturm-Liouville operator from the spectrum, \emph{J. Diff. Eqns.} 282 (2021), 67--103.


\bibitem{Bond25}
Bondarenko, N.~P. Uniform stability of the inverse problem for the non-self-adjoint Sturm-Liouville operator, \emph{Math. Nachr.} 298 (2025), no. 8, 2814--2844. 

\bibitem{Bond20}
Bondarenko, N.~P. Inverse Sturm-Liouville problem with analytical functions in the boundary condition, \emph{Open Math.} 18 (2020), no. 1, 512--528. 


\end{thebibliography}
\end{document}